\documentclass[10pt]{amsart}
\usepackage{geometry}
\usepackage{graphicx}
\usepackage{amssymb}
\usepackage{epstopdf}
\usepackage[colorlinks]{hyperref}
\usepackage[lite]{amsrefs}
\usepackage{mathpazo}
\usepackage{amsmath}
\usepackage{mathrsfs}
\newtheorem{theorem}{Theorem}[section]
\newtheorem{corollary}[theorem]{Corollary}

\newtheorem{proposition}[theorem]{Proposition}
\theoremstyle{definition}

\newtheorem{example}[theorem]{Example}
\numberwithin{equation}{section}

\title{Region crossing change on planar trivalent graphs}
\author{Zhiyun Cheng}
\address{Laboratory of Mathematics and Complex Systems, School of Mathematical Sciences, Beijing Normal University, Beijing 100875, China}
\email{czy@bnu.edu.cn}
\begin{document}
\begin{abstract}
In this short note, we investigate the effect of region crossing change on planar trivalent graphs.
\end{abstract}
\maketitle
\section{Introduction}\label{section1}
A \emph{hypergraph} $H$ consists of a finite set $V=\{v_1, \cdots, v_n\}$ and a family of hyperedges $E=\{e_1, \cdots, e_m\}$, where each $e_i$ is a subset of $V$. For a given hypergraph $H$, the \emph{incidence matrix} of $H$ is a $m\times n$ matrix $M_H=(a_{ij})_{m\times n}$, where
\begin{center}
$a_{ij}=\begin{cases}
1& \text{if $v_j\in e_i$};\\
0& \text{otherwise.}
\end{cases}$
\end{center}
In this note, we are concerned with the rank of this matrix. Throughout this note, all the matrices are considered as elements of $M_{m\times n}(\mathbb{Z}_2)$, which denotes the collection of all $m\times n$ matrices over $\mathbb{Z}_2=\mathbb{Z}/2\mathbb{Z}$. Here we have two examples.

\begin{example}\label{example1.1}
A \emph{signed graph} is a graph $G$ together with a signature, i.e. a function $\sigma: E(G)\to \{0, 1\}$.\footnote{Usually, instead of $\{0, 1\}$, $\sigma(E(G))$ take values in $\{\pm1\}$. Here we make a bit of modification since we are working over $\mathbb{Z}_2$.} A \emph{vertex switching} at a vertex $v\in V(G)$ is an operation which switches all the signs of the edges incident with $v$. A classical result of Harary \cite{Har1955} states that a signed graph is vertex switching equivalent to a signed graph with all the edges colored by 0 if and only if it is balanced. Recall that we say a signed graph is \emph{balanced}, if for any circuit $C$ of $G$, the sum $\sum\limits_{e_i\in E(C)}\sigma(e_i)=0$ $($mod 2$)$.

A natural question arises: for a fixed underlying graph $G$, how many signed graphs are there modulo vertex switchings? This question can be reinterpreted from the viewpoint of hypergraph.

With a given graph $G$, one constructs a hypergraph $H$ as follows: each edge of $G$ corresponds to a vertex of $H$ and each vertex of $G$ corresponds to a hyperedge of $H$. More precisely, suppose that $E(G)=\{e_1, \cdots, e_n\}$, then $H$ has $n$ vertices $\{v'_1, \cdots, v'_n\}$. For a vertex $v_i\in V(G)$ incident with $e_{i_1}, \cdots, e_{i_s}$, we add a hyperedge $e'_i=\{v'_{i_1}, \cdots, v'_{i_s}\}$ to $H$. Notice that a signature $\sigma$ corresponds to a row vector $(\sigma(e_1), \cdots, \sigma(e_n))_{1\times n}$, now the effect of vertex switching on $v_{i_1}, \cdots, v_{i_k}$ can be directly read from the row vector $\sum\limits_{j=1}^kr_{i_j}$. Here $r_i$ denotes the row vector corresponding to $e'_i$ in $M_H$. It follows immediately that the number of non-equivalent signed graphs equals $2^{n-\text{rank}(M_H)}$. In \cite{NRS2015}, it was proved that for a loopless connected graph $G$, rank$(M_H)=|V(G)|-1$. It follows that the number of non-equivalent signed graphs equals $2^{|E(G)|-|V(G)|+1}$.

Note that in this case, 1 appears exactly twice in each column of $M_H$, since each edge has two endpoints.
\end{example}

\begin{example}\label{example1.2}
Consider a link diagram $D$ on the plane, a \emph{region crossing change} on a region $R\in R^2\backslash D$ defines a new link diagram obtained from $D$ by switching all the crossing points on the boundary of $R$. It was proved by Ayaka Shimizu \cite{Shi2014} that for any crossing point $c$ of a knot diagram, there exist some regions such that taking region crossing changes on these regions switches only $c$ but preserves all other crossing points. Obviously, this leads to the result that region crossing change is an unknotting operation for knot diagrams. However, this result does not hold for links with more than one components in general (consider the standard link diagram of Hopf link). For a given link diagram $D$, by ignoring all the crossing information one obtains a 4-valent planar graph, which is called the \emph{link projection} of $D$.

A natural question is, for a fixed link projection, how many link diagrams are there modulo region crossing changes? Once again, we can translate this question to another question which concerns the rank of the incidence matrix of a hypergraph.

Let us use $\{c_1, \cdots, c_n\}$ and $\{R_1, \cdots, R_m\}$ to denote the set of crossing points and the set of regions of a link projection, respectively. Now we define a hypergraph $H$, where each crossing $c_i$ corresponds to a vertex $v_i$ and each region $R_i$ corresponds to a hyperedge $e_i$, such that if $c_{i_1}, \cdots, c_{i_j}$ are on the boundary of $R_i$ then we set $e_i=\{v_{i_1}, \cdots, v_{i_j}\}$. Notice that each link diagram can be encoded into a row vector $M_{1\times n}(\mathbb{Z}_2)$\footnote{For example, one can fix an orientation for each knot component first, then use 0 and 1 to represent positive and negative crossing points respectively.} and the effect of taking region crossing changes on $R_{i_1}, \cdots, R_{i_k}$ can be easily read from the sum $\sum\limits_{j=1}^kr_{i_j}$. Here $r_i$ denotes the row vector of $M_H$ corresponding to $e_i$. It is routine to check that the number of non-equivalent link diagrams equals $2^{n-\text{rank}(M_H)}$. Therefore, in order to obtain the number of non-equivalent link diagrams, it suffices to find out rank$(M_H)$, which has been calculated in \cite{CG2012}. Actually, if a non-split link diagram consisting of $c$ knot components has crossing number $n$, then we have rank$(M_H)=n-c+1$. By setting $c=1$, we recover Ayaka Shimizu's result mentioned above. Based on this, the author proved that region crossing change is an unknotting operation on a link diagram if and only if this link is proper, i.e. $\sum\limits_{j\neq i}lk(K_i, K_j)=0$ $($mod 2$)$ for each component $K_i$ \cite{Che2013}. In particular, it follows that whether region crossing change is an unknotting operation on a link diagram is independent of the choice of the diagram. For example, for any chosen link diagram of the Hopf link, one cannot transform it into a link diagram representing the trivial link via region crossing changes.

We remark that, if a link diagram is reduced, i.e. there is no nugatory crossing, then 1 appears exactly four times in each column of $M_H$, since locally each crossing point lies on the boundaries of four regions.
\end{example}

The purpose of this note is to investigate the effect of region crossing change on planar trivalent graphs. Suppose we are given a planar trivalent graph $G$ embedded on the plane. For simplicity, let us assume that $G$ is connected and it contains no bridge. As a corollary, $G$ is loopless and locally the three regions around each vertex are mutually distinct. Therefore each region is a $n$-gon for some positive integer $n$, and one of them is unbounded. If $n$ is even, then we say this region is an \emph{even} region, otherwise we call it an \emph{odd} region. By a \emph{state} of $G$, we mean a function $s: V(G)=\{v_1, \cdots, v_n\}\to\{0, 1\}$, thus a state can be represented by a row vector $(s(v_1), \cdots, s(v_n))$. Similar to Example \ref{example1.2}, a \emph{region crossing change} on a region defines a new state by switching $s(v_i)$ for all the vertices $v_i$ lying on the boundary of this region. We say two states are \emph{equivalent} if they are related by a sequence of region crossing changes.

As a natural question, we want to know for a fixed planar trivalent graph $G$, how many non-equivalent states does $G$ have. In this paper, we establish the following result, which provides an answer to this question.

\begin{theorem}\label{theorem1.3}
If $G$ is a planar trivalent graph of order $n$ $(n\geq4)$, and having $m$ regions, then
\begin{center}
$s(G)=\begin{cases}
2^{n-m+2}& \text{if all the regions are even;}\\
2^{n-m+1}& \text{if there exist odd regions but no adjacent odd regions, and }\phi(\mathcal{G})\subseteq\mathcal{H};\\
2^{n-m}& \text{otherwise.}
\end{cases}$
\end{center}
Here $s(G)$ denotes the number of non-equivalent states of $G$, and $\mathcal{H}=\{e, (13)\}\subset S_3$.
\end{theorem}

The definition of the group homomorphism $\phi: \mathcal{G}\to S_3$ can be found in Section \ref{section2}. The assumption $n\geq4$ excludes the possibility that $G$ is the $\theta$-graph, which obviously has two non-equivalent states, although all the three regions are even. The reason why we exclude the $\theta$-graph is, in this case, $m=3>2=n$. Actually, since $2|E(G)|=3|V(G)|=3n$, the Euler Identity implies that $m=\frac{n}{2}+2$. It follows that $m\geq4$ if and only if $n\geq4$ and in this case we always have $n\geq m$.

As an analogue of Example \ref{example1.1} and Example \ref{example1.2}, the calculation of $s(G)$ can be translated into the problem of calculating the rank of an incidence matrix $M_H$ for a certain hypergraph $H$. Let us use $\{e_1, \cdots, e_l\}, \{R_1, \cdots, R_m\}$ and $\{v_1, \cdots, v_n\}$ to denote the set of edges, the set of regions and the set of vertices of $G$, respectively. As we mentioned above, here $l, m$ are both determined by $n$, i.e. $l=\frac{3n}{2}$ and $m=\frac{n}{2}+2$. Now we construct a hypergraph $H$ such that $V(H)=V(G)$ and each region $R_i$ corresponds to a hyperedge $r_i$, where $r_i=\{v_{i_1}, \cdots, v_{i_k}\}$ if $v_{i_1}, \cdots, v_{i_k}$ lie on the boundary of $R_i$. We will often abuse our notation by using $r_i$ to denote both a hyperedge of $H$ and the corresponding row vector of $M_H$. Since a state $s$ of $G$ corresponds to a row vector $s=(s(v_1), \cdots, s(v_n))$, taking region crossing changes on $\{R_{i_1}, \cdots, R_{i_k}\}$ yields a new state $s+\sum\limits_{j=1}^kr_{i_j}$. Therefore, for a given state of $G$, there are totally $2^{\text{rank}(M_H)}$ distinct states can be obtained from the given one by taking region crossing changes. It follows that the number of non-equivalent states of $G$ equals $2^{n-\text{rank}(M_H)}$. As a result, in order to prove Theorem \ref{theorem1.3}, it suffices to show that
\begin{center}
$\text{rank}(M_H)=\begin{cases}
m-2& \text{if all the regions are even;}\\
m-1& \text{if there exist odd regions but no adjacent odd regions, and }\phi(\mathcal{G})\subseteq\mathcal{H};\\
m& \text{otherwise.}
\end{cases}$
\end{center}

The followings are some examples.
\begin{figure}[h]
\centering
\includegraphics{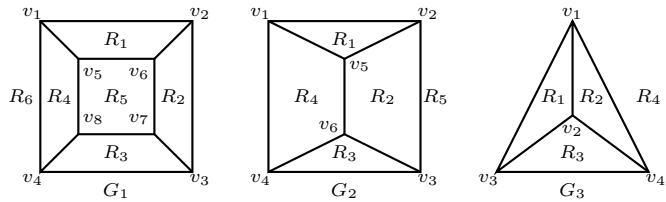}\\
\caption{Three planar trivalent graphs}\label{figure1}
\end{figure}

Here we have three planar trivalent graphs $G_1, G_2, G_3$ illustrated in Figure \ref{figure1}, which correspond to the three cases mentioned in Theorem \ref{theorem1.3}. According to our definition, the incidence matrices of $H_1, H_2$ and $H_3$ are listed below.

$M_{H_1}=\begin{pmatrix}
1&1&0&0&1&1&0&0\\
0&1&1&0&0&1&1&0\\
0&0&1&1&0&0&1&1\\
1&0&0&1&1&0&0&1\\
0&0&0&0&1&1&1&1\\
1&1&1&1&0&0&0&0
\end{pmatrix} M_{H_2}=\begin{pmatrix}
1&1&0&0&1&0\\
0&1&1&0&1&1\\
0&0&1&1&0&1\\
1&0&0&1&1&1\\
1&1&1&1&0&0
\end{pmatrix} M_{H_3}=\begin{pmatrix}
1&1&1&0\\
1&1&0&1\\
0&1&1&1\\
1&0&1&1
\end{pmatrix}$
Direct calculation shows that rank$(M_{H_1})=$ rank$(M_{H_2})=$ rank$(M_{H_3})=4$.

Note that $M_H$ is a matrix of size $m\times n$ $(m\leq n)$, and 1 appears three times in each column of $M_H$, since $G$ is trivalent.

The outline of the paper is as follows. In Section \ref{section2}, we study the so-called $\mathbb{Z}_3$-coloring of planar trivalent graphs and its reduction to a $\mathbb{Z}_2$-coloring. Section \ref{section3} is devoted to give a proof of Theorem \ref{theorem1.3}, based on a calculation of rank$(M_H)$. In Section \ref{section4}, we provide a criterion for triangulating $S^2$ with exactly two odd vertices.

\section{$\mathbb{Z}_3$-coloring and $\mathbb{Z}_2$-coloring of a triangulation}\label{section2}
Let $G$ be a planar trivalent graph embedded on the plane. As before, let us use $\{v_1, \cdots, v_n\}$, $\{e_1, \cdots, e_l\}$ and $\{R_1, \cdots, R_m\}$ to denote $V(G), E(G)$ and the set of all regions respectively. Note that the dual graph $G'$ is a triangulation of the plane, of which one triangle is unbounded. Equivalently, $G'$ also can be regarded as a triangulation of $S^2$. Let $\{v_1', \cdots, v_m'\}, \{e_1', \cdots. e_l'\}$ and $\{R_1', \cdots, R_n'\}$ be $V(G'), E(G')$ and the region set of $G'$, respectively. Throughout this paper, we will frequently switch our focus from $G$ to $G'$ and vice verse, which will not cause confusion since they determine each other mutually.

By a \emph{$\mathbb{Z}_3$-coloring} of the triangulation $G'$, we mean a function $c$ from $\{v_1', \cdots, v_m'\}$ to a 3-element set of colors $\{1, 2, 3\}$ such that adjacent vertices always receive distinct colors. In other words, the three vertices of each triangle are precisely colored by 1, 2 and 3 respectively. It is worthy to point out that, unlike the Tait coloring, which assigns a color from $\{1, 2, 3\}$ to each edge such that the edges of different colors are incident at each vertex, our $\mathbb{Z}_3$-coloring does not always exist. For example, if the degree of a vertex $v_i'$ is odd (we simply say $v_i'$ is \emph{odd}), then such a $\mathbb{Z}_3$-coloring does not exist. It is not difficult to observe that converse also holds, i. e. a triangulation $G'$ admits a $\mathbb{Z}_3$-coloring if and only if all the vertices of it are even. In order to see this, it is sufficient to notice that if a triangle $R_i'$ has been colored properly, the coloring of any other triangle $R_j'$ which shares an edge with $R_i'$ has the only option. Therefore, by an arbitrary choice of a coloring on a triangle, one can extend this colorings to all other triangles uniquely. The key point here is, the two coloring extending around an even vertex are consistent. Hence there is no conflict in the extending. Figure \ref{figure2} explains what will happen if one slides a path over an even vertex or an odd vertex.

\begin{figure}[h]
\centering
\includegraphics{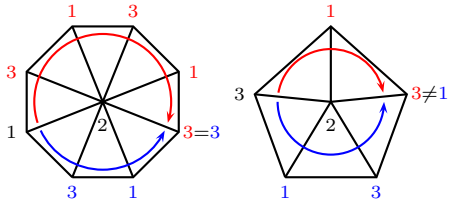}\\
\caption{Sliding over a vertex}\label{figure2}
\end{figure}

When some vertices of $G'$ are odd, there exists no global $\mathbb{Z}_3$-coloring. However, we can still discuss the $\mathbb{Z}_3$-coloring locally. Consider a sequence of triangles $R_{i_1}', \cdots, R_{i_k}'$ such that $R_{i_j}'$ and $R_{i_{j+1}}'$ are adjacent, here $1\leq j\leq k-1$. Then a coloring of $R_{i_1}'$ induces a unique coloring of $R_{i_k}'$ with respect to this sequence. In particular, if $R_{i_1}'$=$R_{i_k}'$, i. e. $R_{i_1}', \cdots, R_{i_k}'$ are dual to a circuit in $G$, then this sequence induces a recoloring of $R_{i_1}'$.

For a given planar trivalent graph $G$, without loss of generality, we assume that the regions $R_1, \cdots, R_k$ $(k\geq2)$ are odd and other regions $R_{k+1}, \cdots, R_m$ are even. Furthermore, we also assume that any two odd regions are not adjacent. Suppose one of the three regions around $v_1$ is odd, say $R_1$. Note that since no pair of odd regions are adjacent, the other two regions must be even, say $R_i, R_j$ ($i, j\geq k+1$). Denote the vertices of $G'$ corresponding to $R_1, \cdots, R_m$ by $v_1', \cdots, v_m'$, then a circuit $l: v_1=v_{i_1}, \cdots, v_{i_k}=v_1$ that begins and ends at $v_1$ corresponds to a sequence of adjacent triangles $R_1'=R_{i_1}', \cdots, R_{i_k}'=R_1'$, which begins and ends at $R_1'$. Fix a $\mathbb{Z}_3$-coloring of the triangle $R_1'$ such that the vertex $v_1'$ is colored by 2. Now suppose we are given a circuit that begins and ends at $v_1$, the corresponding sequence of adjacent triangles induces a recoloring of $R_1'$. As we have seen, such a recoloring is preserved if one slides this sequence over an even vertex of $G'$. It follows immediately that, with a given coloring of $R_1'$ such that the vertex $v_1'$ is colored by 2, the recoloring defines a group homomorphism
\begin{center}
$\phi: \mathcal{G}=\pi_1(S^2\setminus\{v_1', \cdots, v_k'\})\to S_3$,
\end{center}
here $S_3$ denotes the symmetric group of degree 3. This homomorphism was named as \emph{even obstruction map} in \cite{Fis1977}, but here we would like use the name \emph{coloring monodromy}, following \cite{Izm2015}.

Now we turn to the $\mathbb{Z}_2$-coloring of the triangulation $G'$. Roughly speaking, a $\mathbb{Z}_2$-coloring can be obtained from a $\mathbb{Z}_3$-coloring by replacing 1 and 3 with 1 and replacing 2 with 0. More precisely, a \emph{$\mathbb{Z}_2$-coloring} of the triangulation $G'$ is a map from the vertex set $V(G')$ to $\mathbb{Z}_2=\{0, 1\}$, such that for each triangle, the sum of the colors of the three vertices equals 0 (mod 2). It is evident that each $\mathbb{Z}_3$-coloring corresponds to a $\mathbb{Z}_2$-coloring. However, the converse is not true. As an example, the dual graph of the trivalent planar graph $G_2$ in Figure \ref{figure1} does not admit a $\mathbb{Z}_3$-coloring. However, the coloring which assigns 0 to $R_1, R_3$ and 1 to the rest regions suggests that it admits a $\mathbb{Z}_2$-coloring.

\section{The proof of the main theorem}\label{section3}
We first explain how to relate the $\mathbb{Z}_2$-coloring of $G'$ to the region crossing change problem on $G$. Recall that a $\mathbb{Z}_2$-coloring assigns 0 or 1 to each vertex of $G'$ such that the three vertices of each triangle are colored by either $\{0, 0, 0\}$ or $\{0, 1, 1\}$. Obviously, if the vertices of one triangle of $G'$ is colored by $\{0, 0, 0\}$, then all vertices of $G'$ must be colored by 0. We call this kind of coloring the \emph{trivial coloring}, which always exists. What we concern about is the number of nontrivial colorings.

According to the coloring rule, for a nontrivial coloring, if one vertex in $G'$ is colored by 0, then all neighbors of it have to be colored by 1, otherwise one obtains the trivial coloring. On the other hand, if a vertex in $G'$ has color 1, then the neighbors of it must be colored by 0 and 1 alternatively. It follows immediately that if the degree of a vertex in $G'$ is odd, it only can be assigned with 0. For example, if there exist two adjacent odd vertices in $G'$, then $G'$ admits no nontrivial $\mathbb{Z}_2$-coloring.

Now we give the proof of Theorem \ref{theorem1.3}.
\begin{proof}
Denote the nullity of $M_H^T$ by $t$, here $M_H^T$ denotes the transpose of $M_H$. It suffices to show that
\begin{center}
$t=\begin{cases}
2& \text{if all the regions are even;}\\
1& \text{if there exist odd regions but no adjacent odd regions, and }\phi(\mathcal{G})\subseteq\mathcal{H};\\
0& \text{otherwise.}
\end{cases}$
\end{center}
Suppose we are given a solution $\textbf{x}=(x_1, \cdots, x_m)\in M_{1\times m}(\mathbb{Z}_2)$ to the equation $\textbf{x}M_H=0$. In other words, we have $\sum\limits_{i=1}^mx_ir_i=0$. It follows immediately that assigning $x_i$ to vertex $v_i'$ provides a $\mathbb{Z}_2$-coloring of the triangulation $G'$. Therefore, what we need to do is to calculate the number of ``linearly independent" $\mathbb{Z}_2$-coloring of $G'$. We consider three cases.
\begin{enumerate}
  \item Case 1: all vertices of $G'$ are even. In this case, one can freely choose a triangle and color the three vertices by $\{0, 1, 1\}$, $\{1, 0, 1\}$, or $\{1, 1, 0\}$. Now each coloring can be extended to give a global $\mathbb{Z}_2$-coloring of $G'$, since there exists no odd vertex. Note that the third coloring can be represented as the sum of the first two colorings. Obviously, the first two coloring are linearly independent. It follows that in this case the solutions to the equation $\textbf{x}M_H=0$ form a 2-dimensional vector space.
  \item Case 2: there exists two adjacent odd vertices. As we mentioned before, in this case, $G'$ admits no nontrivial $\mathbb{Z}_2$-coloring. As a result, the equation $\textbf{x}M_H=0$ has only zero solution. Thus, the nullity of $M_H^T$ equals zero.
  \item Case 3: there exist some adjacent odd vertices but no pair of them are adjacent. As above, if there exists no nontrivial $\mathbb{Z}_2$-coloring, then the nullity of $M_H^T$ equals zero. Otherwise, all odd vertices have to be colored by 0, therefore the equation $\textbf{x}M_H=0$ has only one nontrivial solution. Choose a triangle $R_1'$ such that one vertex of it, say $v_1'$, is an odd vertex. Fix a $\mathbb{Z}_3$-coloring of $R_1'$ such that $v_1'$ is assigned with 2. Clearly, the triangulation $G'$ admits a nontrivial $\mathbb{Z}_2$-coloring if and only if the recoloring induced by any circuit of $G$ that begins and ends at $v_1$ also assigns 2 to $v_1'$. In other words, the coloring monodromy take values in $\mathcal{H}=\{e, (13)\}\subset S_3$. Under this assumption, the nullity of $M_H^T$ equals one.
\end{enumerate}
The proof is finished.
\end{proof}

\section{Triangulation of $S^2$ with two odd vertices}\label{section4}
Now we consider the special case that $V(G')$ contains exactly two odd vertices, say $v_1'$ and $v_2'$. According to our discussion above, if $c: \{v_1', \cdots, v_m'\}\to\mathbb{Z}_2$ is a nontrivial coloring then $c(v_1')=c(v_2')=0$. Let us assign 0 to $v_1'$, then any path $p$ connecting $v_1'$ and $v_2'$ induces a color for $v_2'$, denoted by $c_p(v_2')$. If this induced color $c_p(v_2')=0$, then we say $v_1'$ and $v_2'$ are \emph{compatible} with respect to $p$. Otherwise, we say $v_1'$ is \emph{incompatible} with $v_2'$ with respect to $p$. Since there are only two odd vertices, this induced color $c_p(v_2')$ actually does not depend on the choice of $p$. If $v_1', v_2'$ are incompatible, then $G'$ admits no nontrivial $\mathbb{Z}_2$-coloring. The following result tells us that this never happens.

\begin{proposition}\label{proposition4.1}
Let $G'$ be a triangulation of $S^2$ with exactly two odd vertices, then $G'$ admits exactly one nontrivial $\mathbb{Z}_2$-coloring.
\end{proposition}

\begin{proof}
Without loss of generality, we use $v_1', v_k'$ to denote the two odd vertices, and $p: v_1', \cdots, v_k'$ to denote a $v_1'-v_k'$ path. It is sufficient to show that $c_p(v_k')=0$.

Let us assume that this is not true, i.e. $c_p(v_k')=1$. First, we claim that there exists a new path $p'$ connecting $v_1'$ and $v_k'$ such that $c_{p'}(v_i')=1$ $(2\leq i\leq k)$. Since $c(v_1')=0$, we know that $c_p(v_2')=1$. Walking along $p$ until we meet the first vertex with color 0, say $v_s'$, where $2<s<k$. Since $c_p(v_s')=0$, all its neighbors should have color 1, in particular $c_p(v_{s-1}')=c_p(v_{s+1}')=1$. By replacing $p: v_1', \cdots, v_{s-1}', v_s', v_{s+1}', \cdots, v_k'$ with $p': v_1', \cdots, v_{s-1}', v_{s_1}', \cdots, v_{s_t}', v_{s+1}', \cdots, v_k'$, where $v_{s_1}', \cdots, v_{s_t}'$ are neighbors of $v_s'$ which connects $v_{s-1}'$ and $v_{s+1}'$. Now the new path $p'$ has one less vertex with color 0. Repeating this process until we obtain a desired path. For convenience sake, let us assume that the original path $p: v_1', \cdots, v_k'$ satisfies $c_{p}(v_i')=1$ $(2\leq i\leq k)$.

Consider the planar trivalent graph $G$, now the path $p$ corresponds to a curve passing through a sequence of regions $R_1, \cdots, R_k$. This curve intersects each common edge of $R_i$ and $R_{i+1}$ $(2\leq i\leq k-2)$ transversely. By replacing each of these intersection points with a small circle, one obtains a new planar trivalent graph $G_1$, see Figure \ref{figure3}. The key point here is, due to the assumption $c_{p}(v_i')=1$ $(2\leq i\leq k)$, the new trivalent graph has exactly two odd regions which are adjacent to each other. In Figure \ref{figure3} these two regions are denoted by $R_{k-1}^1$ and $R_{k-1}^2$. However, it known that for any triangulation of $S^2$, if there are only two odd vertices, they can not be adjacent \cite{Fis1978,Izm2015}. This contradiction means our assumption is incorrect and hence $v_1', v_k'$ are compatible. The proof is finished.
\begin{figure}
\centering
\includegraphics{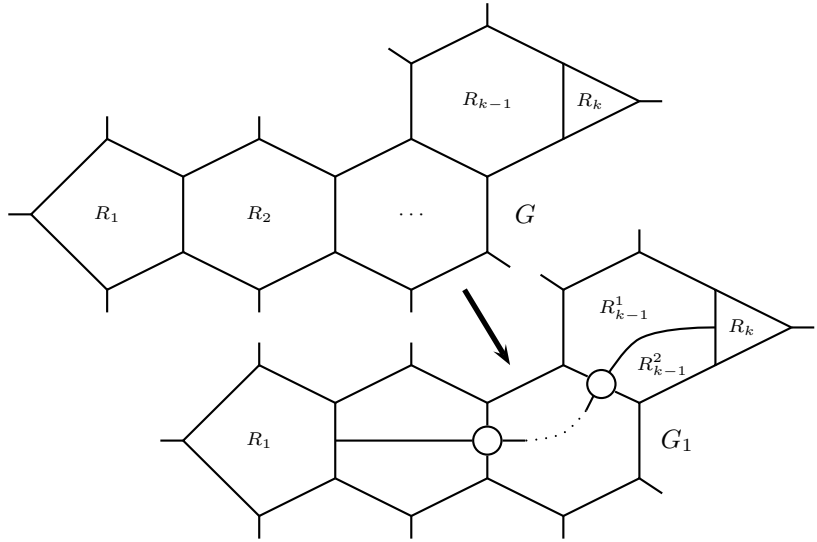}\\
\caption{A new trivalent graph $G_1$ obtained from $G$ }\label{figure3}
\end{figure}
\end{proof}

As a corollary, we have the following result.

\begin{corollary}\label{corollary4.2}
If $G$ is a planar trivalent graph of order $n$, and there are exactly two odd regions among all the $m$ regions, then $s(G)=2^{n-m+1}$.
\end{corollary}

Based on his previous work \cite{Fis1973}, Fisk noticed that if a surface admits a triangulation with two odd vertices and they are adjacent, then this triangulation has no four coloring \cite[Lemma 1]{Fis1978}. Thanks to the celebrated work of Appel and Haken \cite{App1977} (see also \cite{RSS1997}), we know that the four-color theorem holds for $S^2$. Therefore, if a triangulation of $S^2$ has only two odd vertices, then they cannot be adjacent. Recently, this result was reproved by Izmestiev \cite{Izm2015} using coloring monodromy. The main idea of Izmestiev's proof can be sketched as follows. Consider $\mathbb{Z}_3$-colorings of the triangulation $G'$, there exists a homomorphism from the fundamental group of the complement of the set of odd vertices to $S_3$. If the two odd vertices $\{v_1', v_k'\}$ are adjacent, we choose a triangle containing these two odd vertices as the beginning triangle, then the images of the two loops around these two odd vertices generates $S_3$. However, this is impossible since $\pi_1(S^2\backslash\{v_1', v_k'\})\cong\mathbb{Z}$, which is abelian.

Proposition \ref{proposition4.1} tells us that if a triangulation of $S^2$ has exactly two odd vertices, then these two odd vertices not only can not be adjacent, actually they must be compatible. It is interesting to point out that the converse of this result also holds.

More precisely, let us consider a convex $(2m+1)$-gon centered at $u_0'$ and a convex $(2n+1)$-gon centered at $v_0'$ on $S^2$, and label the vertices by $\{u_1', \cdots, u_{2m+1}'\}$ and $\{v_1', \cdots, v_{2n+1}'\}$ respectively. Adding edges $u_0'u_i'$ $(1\leq i\leq 2m+1)$ and edges $v_0'v_j'$ $(1\leq j\leq 2n+1)$ divides these two polygons into $2m+1$ and $2n+1$ triangles. Suppose these two polygons are connected by a sequence of adjacent triangles such that consecutive triangles share an edge, and $u_0'$ and $v_0'$ are compatible with respect to this sequence of triangles. See Figure \ref{figure4} for an example, where all vertices with color 0 are indicated. Now, we obtain a triangulation $T$ of a polygon with only two interior vertices, both of which are of odd degree.

\begin{figure}[h]
\centering
\includegraphics{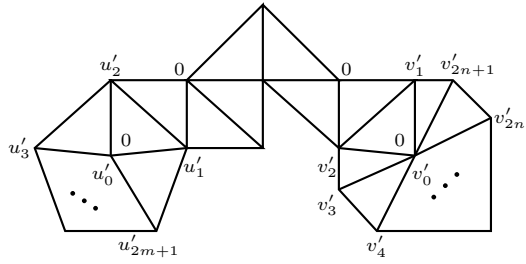}\\
\caption{Two compatible odd vertices induced by a sequence of triangles}\label{figure4}
\end{figure}

\begin{proposition}\label{proposition4.3}
The triangulation $T$ can be extended to give a triangulation of $S^2$ with only two odd vertices.
\end{proposition}
\begin{proof}
The main idea of the proof is adding some triangles to the triangulated polygon to obtain a bigger triangulated polygon, such that all the vertices are even, except the two interior vertices $u_0'$ and $v_0'$. It is known that a $k$-gon admits a triangulation with all vertices of even degree if and only if $k$ is divisible by 3 \cite[Theorem 3]{Izm2015}. Therefore, if we can show that the number of edges on the boundary of the new polygon is divisible by 3, then the complement of the new polygon admits a triangulation with all vertices of even degree. Together with the triangulation of the new polygon, we obtain a triangulation of $S^2$ with only two odd vertices.

\begin{figure}[h]
\centering
\includegraphics{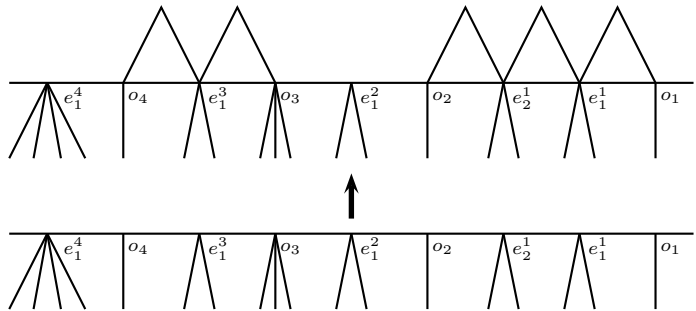}\\
\caption{Adding triangles}\label{figure5}
\end{figure}

In order to explain how to add triangles, let us temporarily use $o_1, \cdots, o_{2k}$ to denote the odd vertices on the boundary of the polygon and use $e^{i}_1, \cdots, e^{i}_{h_i}$ to denote the even vertices between $o_i$ and $o_{i+1}$. In particular, $e^{2k}_1, \cdots, e^{2k}_{h_{2k}}$ denote the even points between $o_{2k}$ and $o_1$. For each odd $i$, we attach $h_i+1$ triangles along the edges $o_ie^i_1, e^i_1e^i_2, \cdots, e^i_{h_i-1}e^i_{h_i}, e^i_{h_i}o_{i+1}$. An example is locally sketched in Figure \ref{figure5}. Finally, we obtain a triangulation of a bigger polygon with only two odd vertices $u_0', v_0'$. It is easy to find that the number of edges on the boundary of this new polygon is equal to $\sum\limits_{i=1}^{2k}(h_i+1)+\sum\limits_{j=1}^{k}(h_{2j-1}+1)$. Let us use $N(T)$ to denote this number. Now what we need to prove is $3|N(T)$.

We remark that whether $N(T)$ is divisible by 3 does not depend on the choice of the first odd vertex. Actually, if we choose $o_2$ as the first odd vertex, the number of edges on the boundary of the new polygon equals $\sum\limits_{i=1}^{2k}(h_i+1)+\sum\limits_{j=1}^{k}(h_{2j}+1)$, which is divisible by 3 exactly when $N(T)$ is divisible by 3.

Without loss of generality, we assume that the first triangle attaches the $(2m+1)$-gon along the edge $u_1'u_2'$ and the last triangle attaches the $(2n+1)$-gon along the edge $v_1'v_2'$. Let us draw the sequence of triangles in the shape of a rectangle, see Figure \ref{figure6} for an example.
\begin{figure}[h]
\centering
\includegraphics{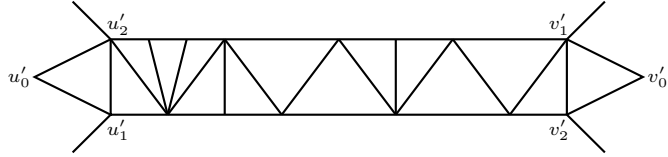}\\
\caption{A sequence of adjacent triangles}\label{figure6}
\end{figure}

Choose a vertex $p\notin\{u_1', u_2', v_1', v_2'\}$ from the rectangle, we claim that the degree of $p$ can be reduced by 2 each time, until it is equal to or less than 4.
\begin{enumerate}
  \item If deg$(p)=s\geq6$ and $p$ is colored by 0, then the neighbors of it, say $q_1, \cdots, q_s$, all have color 1. Deleting the edges $pq_3$ and $pq_4$ yields a new triangulated polygon $T'$. It is easy to observe that $u_0'$ and $v_0'$ are still compatible and $N(T)-N(T')=3$. See Figure \ref{figure7}.
  \begin{figure}[h]
  \centering
  \includegraphics{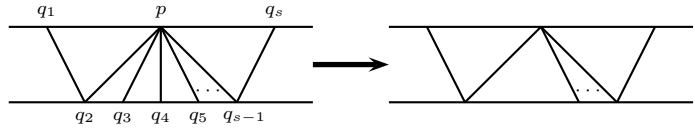}\\
  \caption{Deleting two edges}\label{figure7}
  \end{figure}
  \item If deg$(p)=s\geq6$ and $p$ is colored by 1, then the neighbors $q_1, \cdots, q_s$ must be colored by 0 and 1 alternatively. In this case, we can still delete the edges $pq_3$ and $pq_4$ to obtain a new triangulated polygon $T'$. Now $u_0'$ and $v_0'$ are still compatible. Recall that the choice of the first odd vertex $o_1$ is not important,  by choosing $q_3$ as the first odd vertex, it is not difficult to find that the equation $N(T)-N(T')=3$ also holds in this case.
  \item If deg$(p)=5$, let us use $q_1, \cdots, q_5$ to denote the neighbors of $p$. A key point here is, no matter what color does $p$ have, $q_1, q_5$ receive the same color and $q_2, q_4$ receive the same color. By contracting the edges $q_2q_3$ and $q_3q_4$, we identify the edges $pq_2, pq_4$ with $pq_3$ and obtain a new triangulated polygon $T'$, in which $p$ has three neighbors $q_1, q_3', q_5$. Now $q_3'$ receives the same color as $q_2$ and $q_4$, which guarantees that $u_0'$ and $v_0'$ are still compatible. Next, we need to show that $3|N(T)-N(T')$. We continue the discussion in the following three situations, see Figure \ref{figure8}.
      \begin{enumerate}
        \item[(3.1)] Deg$(q_2)$ and deg$(q_4)$ are both odd. In this case, deg$(q_3')$ is also odd, since $q_3$ is an odd vertex. By choosing $q_2$ as $o_{2k}$ in $T$ (hence $q_3=o_1$ and $q_4=o_2$) and $q_3'$ as $o_{2k-2}$ in $T'$, it is easy to find that $N(T)-N(T')=3$;
        \item[(3.2)] Deg$(q_2)$ is even and deg$(q_4)$ is odd. In this case, $q_3'$ is an even vertex. In $T$, we choose $q_3$ as the first odd vertex $o_1$, and choose the corresponding vertex of $o_3$ in $T'$ as the first vertex. Now we also have $N(T)-N(T')=3$. The case that deg$(q_2)$ is odd and deg$(q_4)$ is even can be proved in a similar manner;
        \item[(3.3)] Deg$(q_2)$ and deg$(q_4)$ are both even. Just like $q_3$, in this case $q_3'$ is also an odd vertex. By choosing $q_3$ and $q_3'$ as the first odd vertex in $T$ and $T'$ respectively, we find that $N(T)-N(T')=3$.
      \end{enumerate}
  \begin{figure}[h]
  \centering
  \includegraphics{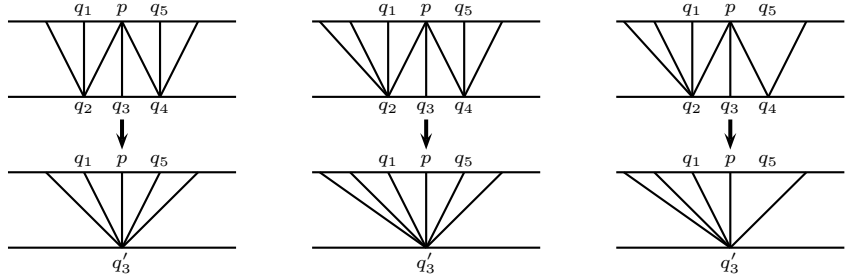}\\
  \caption{Contracting two edges}\label{figure8}
  \end{figure}
\end{enumerate}

If $p\in\{u_1', u_2', v_1', v_2'\}$, we can use the same argument to reduce the degree of it, until it is equal to 4 or 5. As a consequence, it is sufficient to check the following case, see Figure \ref{figure9}.
\begin{figure}[h]
\centering
\includegraphics{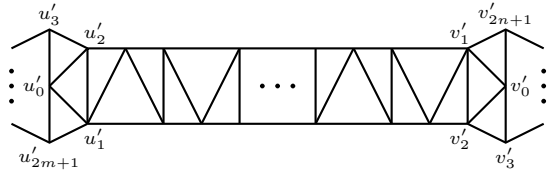}\\
\caption{Two compatible odd vertices connected by a sequence of triangles}\label{figure9}
\end{figure}

Since $u_0'$ are $v_0'$ are compatible, the number of triangles in the rectangle connecting the two polygons has the form $3l$ for some $l\geq0$. It follows that if we choose $u_3'$ as the first odd vertex, then we have $N(T)=2m+2n+3l+m+n=3(m+n+l)$. The proof is completed.
\end{proof}

\section*{Acknowledgement}
Zhiyun Cheng is supported by NSFC 12071034.

\end{document}